\documentclass[a4paper,10pt]{amsart}
\usepackage{amsmath,amsfonts,amssymb, amsthm, xypic}
\usepackage[all]{xy}
\usepackage[left=2cm, right=2cm, top=2cm]{geometry}

\def\N{\mathbb{N}}

\def\Z{\mathbb{Z}}
\def\Q{\mathbb{Q}}

\def\a{\alpha}
\def\cT{\mathcal{T}}
\def\cF{\mathcal{F}}
\def\bd{\mathbf{d}}
\def\m1{\mathbf{1}}

\def\b{\beta}
\def\up{\underline{p}}
\def\O{\mathcal{O}}
\def\ul{\underline{\lambda}}

\def\ga{\gamma}
\def\C{\mathbb{C}}

\def\Im{\text{Im}}

\def\tto{\twoheadrightarrow}

\newtheorem{theo}{\bf{Theorem}}[section]
\newtheorem{lem}[theo]{Lemma}

\newtheorem{prop}[theo]{Proposition}

\newtheorem{conj}[theo]{Conjecture}

\numberwithin{equation}{section}

\title[Kac polynomials for canonical algebras]{Kac polynomials for canonical algebras}
\author{Pierre-Guy Plamondon et Olivier Schiffmann$^\dag$}
\thanks{$\dag$ : partially supported by ANR grant 13-BS01-0001-01}
\date{}

\begin{document}

\begin{abstract}
We prove that the number of geometrically indecomposable representations of fixed dimension vector $\mathbf{d}$ of a canonical algebra $\mathbf{C}$ defined over a finite field $\mathbf{F}_q$ is given by a polynomial in $q$ (depending  on $\mathbf{C}$ and $\mathbf{d}$). We prove a similar result for squid algebras. Finally we express the volume of the moduli stacks of representations of these algebras of a fixed dimension vector in terms of the corresponding Kac polynomials.
\end{abstract}

\maketitle

\vspace{.1in}

\section{Introduction}

In the early 80s, Kac showed that the number of isomorphism classes of absolutely indecomposable representations of a fixed dimension vector $\mathbf{d}$ of a quiver $Q$ over a finite field $\mathbb{F}_q$ is given by the evaluation at $t=q$ of a polynomial $A_{Q,\mathbf{d}}(t) \in \mathbb{Z}[t]$ (depending only on $Q$ and $\mathbf{d}$). This is especially surprising since absolutely indecomposable representations only form a constructible substack of the stack $\textbf{Rep}_{\mathbf{d}}(Q)$ of all representations of $Q$ of dimension $\mathbf{d}$. It was conjectured by
Kac and proved much later by Letellier, Hausel and Rodriguez-Villegas that $A_{Q,\mathbf{d}}(t) \in \mathbb{N}[t]$
(see \cite{K2}, \cite{LHRV}).
When $Q$ has no edge loops, the polynomial $A_{Q,\mathbf{d}}(t)$ carries a lot of representation-theoretic information related to the Kac-Moody Lie algebra $\mathfrak{g}_{Q}$ attached to $Q$ (see \cite{K2}, \cite{HauselKac}).

The aim of this note is to generalize Kac's theorem to the case of a certain family of quivers \textit{with relations}
which play an important role in the representation theory of finite-dimensional algebras. These are the so-called \textit{canonical algebras} introduced by Ringel in \cite{RingelCanonical}. Note that these algebras are of global dimension \textit{two} rather than one, hence the original arguments
of Kac (\cite{K2}) or Hua (\cite{Hua}) fail. On the other hand, these algebras are derived equivalent to certain categories of global dimension one, namely the categories of coherent sheaves on weighted projective lines
introduced by Geigle and Lenzing (see \cite{GL}), for which an analog of Kac's theorem does hold (see \cite[Section 7]{SHiggs} ). Our argument is based on the existence of such derived equivalences, and in fact our main result extends to an arbitrary \textit{almost concealed canonical algebra}, i.e. to the endomorphism algebra of any tilting sheaf on a weighted projective line (for instance, it applies to squid algebras). We refer to Section~3 for precise definitions.

\begin{theo}\label{T:main2} For any almost concealed-canonical algebra $C$ defined over a finite $field k$ and any class $\mathbf{d} \in K^+_0(mod(C))$ there exists a unique polynomial $A_{\mathbf{d},C} \in \Z[x]$ such that for any finite field extension $l$, the number of absolutely indecomposable representations of dimension $\mathbf{d}$ of $C$ over $l$ is equal to $A_{\mathbf{d}, C}(|l|)$. Moreover these polynomials only depend on the combinatorial type of $C$.
\end{theo}

\vspace{.1in}

\section{Recollections on weighted projective lines}

\vspace{.1in}

In this Section we recall the definitions and main properties of Geigle and Lenzing's weighted projective lines, refering to \cite{GL} for details. In this section, we let $k$ be an arbitrary field.

\vspace{.1in}

\paragraph{\textbf{2.1.}} Consider the projective line $\mathbb{P}^1$ defined over $k$. Let $\lambda_1, \ldots, \lambda_N$ be $N$ distinct $k$-points on $\mathbb{P}^1$ and let $p_1, \ldots, p_N$ be integers satisfying $p_i >1$ for all $i$. Without loss of generality, we may assume that $\lambda_1=\infty$ and $\lambda_2=0$. The simplest way to define the category $Coh(\mathbb{X})=Coh(\mathbb{X}_{\underline{p},\underline{\lambda}})$ is by an analogue of Serre's construction
of the category of coherent sheaves on a projective variety by means of the ring of global sections of (powers of) an ample line bundle. Namely, we set
$$S(\underline{p}, \underline{\lambda})=k[X_1, \ldots, X_n]/ I$$
where $I$ is the ideal generated by the elements $X_i^{p_i}-X_2^{p_2}-\lambda_i X_1^{p_1}$ for $i =3, \ldots, N$.
We also consider the abelian group
$$L_{\underline{p}}:= \left(\bigoplus_{i=1}^N \Z \vec{x}_i\right) / J$$
where $J$ is the subgroup generated by the elements $p_i \vec{x}_i= p_j \vec{x_j}$ for all $i, j$. Giving $X_i$ the degree $\vec{x}_i$ yields an $L_{\underline{p}}$-grading on $S_{\underline{p}, \underline{\lambda}}$.
By definition,
$$Coh(\mathbb{X}):= Mod_{gr} S(\underline{p},\underline{\lambda}) / Mod_{gr,0} S(\underline{p},\underline{\lambda})$$
where $Mod_{gr}$ denotes the category of finitely generated $L_{\underline{p}}$-graded modules and $Mod_{gr,0}$ is the Serre subcategory whose objects are of finite length. We refer to \cite{GL} for a more local construction of
the category $Coh(\mathbb{X})$ as a category of modules over a sheaf of hereditary orders over $\mathbb{P}^1$.
The category $Coh(\mathbb{X})$ is abelian and of global dimension one.
Observe that the map $X \mapsto X_1^{p_1}, Y \mapsto X_2^{p_2}$ induces a graded algebra morphism
$k[X,Y] \to S(\up,\ul)$ whose cokernel is of finite dimension. It induces an exact functor $Coh(\mathbb{X}) \to Coh(\mathbb{P}^1)$ which, however, is not faithful.

\vspace{.1in}

\paragraph{\textbf{2.2.}} Let us denote by $Coh^{l.f.}(\mathbb{X})$ the subcategory of $Coh(\mathbb{X})$ defined as the essential image of the subcategory $ Mod^{t.f.}_{gr} S(\underline{p},\underline{\lambda})$ of torsion-free modules under the functor $Mod_{gr}S(\up, \underline{\lambda}) \to Mod_{gr} S(\underline{p},\underline{\lambda}) / Mod_{gr,0} S(\underline{p},\ul)$. We define $Coh_0(\mathbb{X})$ in a similar fashion as the essential image of the subcategory $Mod^{t}S(\underline{p},\underline{\lambda})$ of torsion modules. Objects of $Coh^{l.f.}(\mathbb{X})$ and $Coh_0(\mathbb{X})$ are respectively called vector bundles and torsion sheaves over $\mathbb{X}$. Any object $\mathcal{G} \in Coh(\mathbb{X})$ has a unique maximal subobject
$\mathcal{G}^{tor}$ and the quotient $\mathcal{G}^{l.f.}:= \mathcal{G}/\mathcal{G}^{tor}$ is a vector bundle. In addition there is a (non-canonical) splitting $\mathcal{G} \simeq \mathcal{G}^{l.f.} \oplus \mathcal{G}^{tor}$.  For any pair $\mathcal{G} \in Coh^{(l.f.}(\mathbb{X}),\mathcal{H}\in Coh^{tor}(\mathbb{X})$ we have
$$Hom(\mathcal{H},\mathcal{G})=\{0\}, \qquad Ext^1(\mathcal{G},\mathcal{H})=\{0\}.$$
The category $Coh^{l.f.}(\mathbb{X})$ is equivalent to a category of vector bundles on $\mathbb{P}^1$ equipped with a
suitable parabolic structure at the points $\{\lambda_1, \ldots, \lambda_N\}$, see \cite{Lenzing} for details.

\vspace{.1in}

Let us now briefly describe some important classes of objects in $Coh(\mathbb{X})$. If $M$ is a graded $S(\up, \ul)$-module and $\vec{u} \in L_{\up}$ then we define the shifted module $M[\vec{u}]$ by $M[\vec{u}]_{\vec{v}}=M_{\vec{u}+\vec{v}}$. The grading shift $M \mapsto M[\vec{u}]$ induces an auto-equivalence
of the category $Coh(\mathbb{X})$, also denoted $\mathcal{G} \mapsto \mathcal{G}(\vec{u})$.
The modules $S(\up,\ul)[\vec{u}]$ give rise to objects $\mathcal{O}(\vec{u})$ in $Coh(\mathbb{X})$; these are the line bundles on $\mathbb{X}$, and form the
Picard group $Pic(\mathbb{X})$. We have
$$Hom(\mathcal{O}(\vec{u}), \mathcal{O}(\vec{v}))= S(\up,\ul)_{\vec{v}-\vec{u}}.$$
It is known that any vector bundle on $\mathbb{X}$ admits a filtration whose successive factors are line bundles.
For any $i$ we denote by $S_i^{(0)}$ the torsion sheaf associated to the module $S(\up,\ul) / X_i S(\up,\ul)$, and
put $S_i^{(l)}=S_i^{(0)}(l\vec{x}_i)$. The sheaf $S_i^{(l)}$ only depends up to isomorphism on the class of $l$ in $\Z/ p_i\Z$. Multiplication by $X_i$ yields a short exact sequence
$$\xymatrix{0 \ar[r] & \mathcal{O}(l\vec{x}_i) \ar[r] & \mathcal{O}((l+1)\vec{x}_i) \ar[r] & S_i^{(l)} \ar[r] & 0}.$$
The sheaves $S_i^{(0)}, S_i^{(1)}, \ldots, S_i^{(p_i-1)}$ generate a Serre subcategory $Coh_{0,\lambda_i}(\mathbb{X})$ of $Coh_0(\mathbb{X})$, consisting of sheaves supported at the point $\lambda_i$. The category 
$Coh_{0,\lambda_i}(\mathbb{X})$ is equivalent to the category of nilpotent representations of the cyclic quiver of  type $A_{p_i-1}^{(1)}$ over the field $k$, see e.g. \cite[Lemma~2.4]{SDuke}

\vspace{.1in}

\paragraph{\textbf{2.3.}} Let $\vec{w}=
(N-2) \vec{c}-\sum_i \vec{x}_i \in L_{\up}$, where we have set $\vec{c}=p_i \vec{x}_i$ for any $i$. Then the
functor $\mathcal{G} \mapsto \mathcal{G}(\vec{w})$ is a Serre functor for the category $Coh(\mathbb{X})$, i.e. there are functorial isomorphisms
$$Ext^1(\mathcal{G},\mathcal{H})^* \simeq Hom(\mathcal{H}, \mathcal{G}(\vec{w})), \qquad \forall\; \mathcal{G},\mathcal{H} \in Coh(\mathbb{X}).$$
The auto-equivalence $\mathcal{G} \mapsto \mathcal{G}(\vec{w})$ preserves $Coh^{l.f.}(\mathbb{X})$ and $Coh^{tor}(\mathbb{X})$.

\vspace{.1in}

\paragraph{\textbf{2.4.}} Let $K_0$ stand for the Grothendieck group of the category $Coh(\mathbb{X})$, and let $K_0^+ \subset K_0$ be the cone of classes of objects in $Coh(\mathbb{X})$, and $K_0^{+,tor}$ the cone of torsion classes. We denote by $\overline{\mathcal{G}}$ the class of the sheaf $\mathcal{G}$. The group $K_0$ is equipped with the Euler form
\begin{align*}
\langle\;,\;\rangle~:K_0 \otimes_\Z K_0 & \to \Z\\
\overline{\mathcal{G}} \otimes \overline{\mathcal{H}} &\mapsto \langle \overline{\mathcal{G}}, \overline{\mathcal{H}} \rangle = dim(Hom(\mathcal{G},\mathcal{H}))- dim (Ext^1(\mathcal{G}, \mathcal{H}) )
\end{align*}

The precise description of $K_0$ and the Euler form is given below for the comfort of the reader (see e.g. \cite[Section~5]{SDuke} for details). Consider the $\Z$-module 
$$N_{\up}:= \left(\Z e  \oplus \bigoplus_{i=1}^N \bigoplus_{s\in \Z/p_i\Z} \Z e_{i,s}\right) / J$$
where $J$ is the subgroup generated by the elements $\sum_{s \in \Z/p_i\Z}e_{i,s} =\sum_{s \in \Z/p_j\Z}e_{j,s} $ for $i \neq j$.
For simplicity, we set $\delta =\sum_{s \in \Z/p_i\Z}e_{i,s}$ for any $i$. We equip $N_{\up}$ with a bilinear form defined as follows~:
 $$
\langle e, e \rangle =1, \qquad \langle e, \delta \rangle =1,
\qquad \langle \delta, e \rangle =-1,$$
$$\langle \delta,\delta \rangle =0, \qquad \langle \delta, e_{i,s} 
\rangle =0, \qquad \langle e_{i,s},\delta \rangle =0,
$$
and
$$\langle e, e_{i,s} \rangle =\begin{cases} 1 & {if}\; s=-1\\
0 & {if}\; s \neq -1 \end{cases}$$
$$\langle e_{i,s}, e \rangle =\begin{cases} -1 & {if}\; s=0\\
0 & {if}\; s \neq 0 \end{cases}$$
$$\langle e_{i,s}, e_{i',s'} \rangle =\begin{cases} 1 & {if}
\; i=i',\;s=s'\\
-1 & {if}\; i=i',\;s= s'+1\\
0 & {otherwise} \end{cases}.$$

\begin{lem}\label{L:kzero} The assignement $\overline{\O} \mapsto e, \overline{S_i^{(s)}} \mapsto e_{i,s}$ extends to an isomorphism
$(K_0, \langle,\;\,\rangle)=\simeq (N_{\up}, \langle, \;\, \rangle)$ of $\Z$-modules equipped with bilinear forms .
\end{lem}

%
%

Finally, we define two important functions on the Grothendieck group $K_0$, the \textit{rank} and the \textit{degree}, as follows~:
$$rank(\overline{\O})=1, \qquad rank(\overline{S}_i^{(s)})=0\quad \;\forall \;i,s,$$
$$deg(\overline{\O})=0, \qquad deg(\overline{S}_i^{(s)})=\frac{p}{p_i}\quad\;\forall \;i,s,$$
where $p=l.c.m.(p_1, \ldots, p_N)$. Observe that 
$$rank(\mathcal{G})=rank(\mathcal{G}(\vec{u})), \qquad deg(\mathcal{G}(\vec{u}))=deg(\mathcal{G}) + rank(\mathcal{G})\sum_i u_i \frac{p}{p_i}$$
for any $\mathcal{G}$ and any $\vec{u}=\sum_i u_i \vec{x}_i$. We set 
$$\kappa= deg (\mathcal{O}(\vec{w}))=p(N-2) -\sum_i \frac{p}{p_i}.$$
The number $\kappa$ plays the role, in the present context, of the degree of the canonical bundle. In particular, the \textit{genus} of $\mathbb{X}$ is defined as $g_{\mathbb{X}}=1 + \kappa/2$.

\vspace{.1in}

\paragraph{\textbf{2.5.}} The \textit{slope} of a sheaf $\mathcal{G}$ is defined to be
$$\mu(\mathcal{G})=\frac{deg(\mathcal{G})}{rank(\mathcal{G})} \in \mathbb{Q} \cup \{\infty\}.$$
We say that a sheaf $\mathcal{G} \in Coh(\mathbb{X})$ is \textit{semistable of slope $\nu$} if $\mu(\mathcal{G})=\nu$ and if for any subsheaf $\mathcal{H} \subset \mathcal{G}$ we have $\mu(\mathcal{H}) \leq \nu$. We recall in the following lemma the main properties of semistable sheaves (see e.g. \cite{GL} for details).

\vspace{.1in}

\begin{lem} The following hold~:
\begin{enumerate}
\item[i)] If $\mathcal{G},\mathcal{H}$ are semistable sheaves of slopes $\nu, \sigma$ respectively then
$$\nu > \sigma \Rightarrow Hom(\mathcal{G},\mathcal{H})=\{0\}, \qquad \sigma > \mu + \kappa \Rightarrow Ext^1(\mathcal{G}, \mathcal{H})=\{0\}.$$
\item[ii] Any coherent sheaf $\mathcal{G}$ admits a unique filtration (called the \textit{Harder-Narasimhan filtration})
$$\mathcal{G}=\mathcal{G}_0 \supset \mathcal{G}_1  \supset \cdots \supset \mathcal{G}_{s-1} \supset \mathcal{G}_{s}=0$$
whose factors $\mathcal{H}_i := \mathcal{G}_{i-1}/\mathcal{G}_i$ are semistable and satisfy
$$\mu(\mathcal{H}_1) < \mu(\mathcal{H}_2) < \cdots < \mu(\mathcal{H}_s).$$
\end{enumerate}
\end{lem}

In the notations of the lemma above, we set
$$\mu_{min}(\mathcal{G})=\mu(\mathcal{H}_1), \qquad \mu_{max}(\mathcal{G})=\mu(\mathcal{H}_s).$$ 
 Thus a coherent sheaf $\mathcal{G}$ is a torsion sheaf if and only if it is semistable of slope $\infty$, while it is a vector bundle if and only if $\mu_{max}(\mathcal{G}) <\infty$.

\vspace{.2in}

\section{Tilting sheaves and (almost)(concealed-)canonical algebras}

\vspace{.1in}

\paragraph{\textbf{3.1.}} 
We follow the lines of \cite{Meltzer} for this section. Let $k$ again be any field, and let $\mathbb{X}$ be the weighted projective line defined over $k$, attached to the data $(p_1, \ldots, p_N, \lambda_1, \ldots, \lambda_N)$
A \emph{tilting sheaf} in $Coh(\mathbb{X})$ is a coherent sheaf $T$ such that $Ext^1(T,T) = 0$ and $T$ generates the bounded derived category $\mathcal{D}^b(Coh(\mathbb{X}))$ as a triangulated category.  A \emph{tilting bundle} is a tilting sheaf $T$ which is also a vector bundle.  Note that, since $Coh(\mathbb{X})$ is of global dimension one, all higher extension spaces of $T$ vanish as well.

Given a tilting sheaf $T$, we get a derived equivalence $RHom(T,?):\mathcal{D}^b(Coh(\mathbb{X}))\stackrel{\sim}{\to} \mathcal{D}^b(mod(C_T))$, where $C_T=End(T)$. The algebra $C_T$ is called \emph{almost concealed-canonical}.  If, moreover, $T$ is a tilting bundle, then the algebra $C_T$ is called \emph{concealed canonical}. In any case, $C_T$ is a finite-dimensional algebra of global dimension two. 

Any tilting sheaf $T$ defines a torsion pair $(\mathcal{T}, \mathcal{F})$ in $Coh(\mathbb{X})$, where
$$ \mathcal{T} = \{X\in Coh(\mathbb{X}) \ | \ Ext^1_\mathbb{X}(T,X) = 0\}, \qquad  \mathcal{F} = \{X\in Coh(\mathbb{X}) \ | \ Hom_\mathbb{X}(T,X) = 0\}. $$
The equivalence $\mathcal{D}^b(Coh(\mathbb{X}))\stackrel{\sim}{\to} \mathcal{D}^b(mod(C_T))$ allows us to view the module category of $C_T$ as a subcategory of $\mathcal{D}^b(Coh(\mathbb{X}))$
The next result is proved in \cite[Sect. 8.1]{Meltzer}~:

\vspace{.05in}

\begin{theo}[Meltzer]\label{T: indecan}The following is true~:
\begin{enumerate}
\item[0)] the subcategories $\mathcal{T}$ and $\mathcal{F}[1]$ are extension-closed in $mod(C_T)$,
\item[i)] any indecomposable $C_T$-module belongs to either $\mathcal{T}$ or $\mathcal{F}[1]$ and conversely, any indecomposable object of $\mathcal{T}$ or $\mathcal{F}[1]$ belongs to $mod(C_T)$,
\item[ii)] Any $M \in \mathcal{F}[1]$ is of injective dimension at most one; any $N \in \mathcal{T}$ is of projective dimension at most one,
\item[iii)] for any pair $M \in \mathcal{F}[1]$, $N \in \mathcal{T}$ we have
$$Ext^2(N,M)=Ext^1(N,M)=Hom(M,N)=\{0\},$$
In particular, any $C_T$ module $M$ has a canonical submodule $M^{(1)} \subset M$ such that $M^{(1)} \in \mathcal{F}[1]$, $M^{(2)}:= M/M^{(1)} \in \mathcal{T}$ and $M \simeq M^{(1)} \oplus M^{(2)}$.
\end{enumerate}
\end{theo}

Recall that $K_0$ is the Grothendieck group of the abelian category $Coh(\mathbb{X})$. There is a sequence of isomorphisms $K_0 \stackrel{\sim}{\to} K_0(\mathcal{D}^b(Coh(\mathbb{X}))) \stackrel{\sim}{\to} K_0(\mathcal{D}^b(mod(C_T))) \stackrel{\sim}{\to} K_0(mod(C_T))$. Let us denote by $\psi : K_0 \stackrel{\sim}{\to} K_0(mod(C_T))$ their composition. We will implicitly use the map $\psi$ to identify $K_0(mod(C_T))$ with $K_0$ and hence, thanks to Lemma~\ref{L:kzero}, with $N_{\up}$.
For any additive subcategory $\mathcal{C}$ of $Coh(\mathbb{X})$ or $mod(C_T)$, denote by $K_0^+(\mathcal{C})$ the submonoid of $K_0$ consisting of classes of objects of $\mathcal{C}$. We define the \textit{combinatorial type} of $C_T$ to be equal to the tuple $($\up$, \psi(K_0^+(\mathcal{T})), \psi(K_0^+(\mathcal{F}[1])))$. 

 The following observations will turn out to be useful.

\begin{lem}\label{L:intersect}
 The intersection of $K_0^+(\mathcal{T})$ and $K_0^+(\mathcal{F}[1])$ in $K_0$ is trivial.
\end{lem} 
\begin{proof}
Indeed, no object of $Coh(\mathbb{X})$ has zero class in $K_0$.
\end{proof}

\vspace{.1in}

\begin{lem}\label{L:last} For any fixed $T$, there exists a rational number $\nu$ such that for any $\mathcal{G} \in \cT$ we have
$\mu_{min}(\mathcal{G}) \geq \nu$.
\end{lem}
\begin{proof} We will show that $\nu=\mu_{min}(T)$ verifies the above condition. Indeed, let $\mathcal{G} \in Coh(\mathbb{X})$ and let us assume that $\mu_{min}(\mathcal{G}) < \nu$. We will prove that $Ext^1(T,\mathcal{G})\neq \{0\}$, which implies that $\mathcal{G} \not\in \cT$. Let $\mathcal{H}_1$ be the semistable quotient of $\mathcal{G}$ of slope $\mu_{min}(\mathcal{G})$ (see Section~\textbf{2.5}), and
let us write $T =\bigoplus_i T_i$ for the direct sum decomposition into indecomposable sheaves. It is enough to prove that there
exists $i$ such that $Ext^1(T_i, \mathcal{H}_1) \neq \{0\}$. Because $\mathcal{H}_1$ is semistable of slope $\nu < \mu_{min}(T)$
we have $Hom(T_i, \mathcal{H}_1)=\{0\}$ for any $i$, and hence it is enough to check that there exists $i$ such that $\langle T_i, \mathcal{H}_1\rangle \neq 0$. In fact, we claim that for any $\alpha \in K_0$, $\alpha \neq 0$ there exists $i$ such that $\langle T_i, \alpha \rangle \neq 0$. Indeed, under the derived equivalence $\mathcal{D}^b(Coh(\mathbb{X}))\stackrel{\sim}{\to} \mathcal{D}^b(mod(C_T))$
the sheaves $T_i$ are mapped to the indecomposable projectives $P_i$, and for any module $M \in mod(C_T)$ there exists some $i$ such
that $Hom(P_i, M) \neq \{0\}$, which implies that $\langle P_i, M \rangle \neq 0$. This concludes the proof since the Grothendieck group
and the Euler form are invariant under any derived equivalence between categories of finite homological dimension. 
\end{proof}

\vspace{.1in}

\noindent
\textbf{Remark.} If $T$ is a vector bundle then the same argument yields the existence of a rational number $\eta$ such that
$\mu_{max}(\mathcal{G}) \leq \eta$ for any $\mathcal{G} \in \cF$.

\vspace{.1in}

\begin{lem}\label{L:finitesum} For any fixed $T$, for any $\gamma \in K_0^+$, the set of isomorphism classes of sheaves
$\mathcal{G}$ in $\cT$ (resp. in $\cF$) of class $\gamma$ is finite.\end{lem}
\begin{proof} By Lemma~\ref{L:last}, there exists $\nu$ such that any $\mathcal{G} \in \cT$ satisfies $\mu_{min}(\mathcal{G}) \geq \nu$.
For any $\gamma$, the set of isomorphism classes of sheaves $\mathcal{G} \in Coh(\mathbb{X})$ satisfying $[\mathcal{G}]=\gamma$ and
$\mu_{min}(\mathcal{G}) \geq \nu$ is finite. This implies \textit{a fortiori} the statement of the Lemma concerning the category $\cT$. The
case of the category $\cF$ is slightly more involved since unless $T$ is a vector bundle there is no lower (or upper) bound for the slopes occuring in the Harder-Narasimhan types of sheaves in $\cF$. Let us write $T=T_{vec} \oplus T_{tor}$ where $T_{vec}$ is a vector bundle and $T_{tor}$ a torsion sheaf. By the same argument as in Lemma~\ref{L:last} there exists $\eta$ such that for any \textit{vector bundle}
$\mathcal{G} \in \cF$ we have $\mu_{max}(\mathcal{G}) \leq \eta$. Now, because $Ext^1(T_{tor}, T_{tor})=\{0\}$ we have
$T_{tor}=\bigoplus_{i=1}^N T_{tor,i}$ with $T_{tor,i}$ a torsion sheaf supported at the exceptional point $\lambda_i$ satisfying $Ext^1(T_{tor,i}, T_{tor,i})=\{0\}$. We deduce that there exists $u_i \in \Z/p_i\Z$ such that $[T_{tor,i}] \in \bigoplus_{l \neq u_i} \N [S_i^{(l)}]$
(see Appendix~A). We claim that for any fixed $\gamma$, there exist only finitely many decompositions $\gamma= \gamma_1 + \gamma_2$, with $\gamma_1$ the class of a vector bundle and $\gamma_2 \in \bigoplus_{i,l\neq u_i}\N [S^{(l)}_i]$. To see this, first note that because any vector bundle $V$ is an extension of line bundles and because $\langle \mathcal{L}, [S_i^{(l)}]\rangle \leq 1$ for any line bundle $\mathcal{L}$ and any simple torsion sheaf $S_i^{(l)}$, we have $0 \leq \langle {V}, S_{i}^{(l)}\rangle \leq rk(V)$ for any vector bundle $V$ and any $(i,l)$. It follows that for any sheaf $V \oplus T \in \cF$ of class $\gamma$ with $V$ a vector bundle and $T$ a torsion sheaf, we have 
$$\langle \gamma, S_i^{(l)}\rangle -rk(\gamma) \leq \langle T, S_i^{(l)}\rangle \leq \langle \gamma, S_i^{(l)}\rangle$$
for any $(i,l)$.
\end{proof}

\vspace{.2in}

\paragraph{\textbf{3.2. Examples~:}}  i) Let us choose the tilting bundle $T=\bigoplus_{0 \leq \vec{x} \leq \vec{c}} \mathcal{O}(\vec{x})$. Then, as shown by Geigle and Lenzing (\cite{GL}), $C_T$ is a \emph{canonical algebra}. This class of algebras was first introduced by Ringel in \cite{RingelCanonical}; they may also be directly defined as path algebras of certain quivers modulo some relations. More precisely, the \emph{canonical algebra} $C_T$ is isomorphic to the path algebra of the quiver

$$
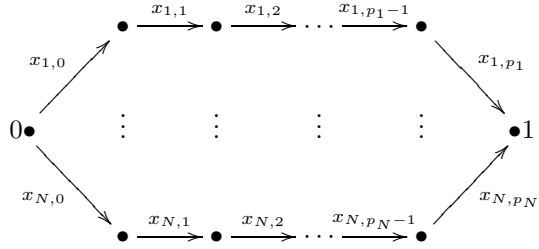
\xymatrix{
& \bullet \ar[r]^-{x_{1,1}} & \bullet \ar[r]^-{x_{1,2}} & \cdots \ar[r]^-{x_{1,p_1-1}} & \bullet \ar[dr]^-{x_{1,p_1}} & &\\
0 \bullet \ar[ru]^-{x_{1,0}} \ar[rd]_-{x_{N,0}} & \vdots & \vdots &\vdots & \vdots& \bullet 1\\
& \bullet \ar[r]^-{x_{N,1}} & \bullet \ar[r]^-{x_{N,2}} & \cdots \ar[r]^-{x_{N,p_N-1}} & \bullet \ar[ur]_-{x_{N,p_N}} &}
$$
modulo the $(N-2)$ relations
$$
 \prod_{j=0}^{p_i} x_{i,p_{i}-j} = \prod_{j=0}^{p_1} x_{1,p_{1}-j} - \lambda_i \prod_{j=0}^{p_2} x_{2,p_{2}-j},
$$
for each $i=3,4,\ldots, N$.

\vspace{.1in}

\noindent
ii) Now let us choose the tilting sheaf $T=\mathcal{O} \oplus \mathcal{O}(\vec{c}) \oplus \bigoplus_{i=1}^N \bigoplus_{j=1}^{p_{i-1}} S_{i}^{(j)}$. In that case, $C_T$ is a squid algebra (as first introduced by Brenner and Butler, \cite{BB}). By definition, it is the path algebra of following quiver
$$
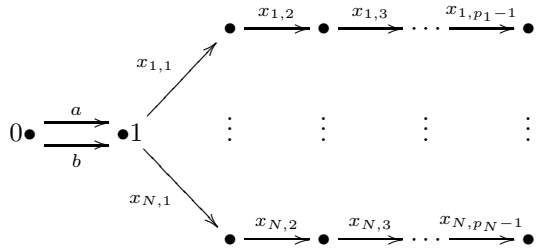
\xymatrix{
& &\bullet \ar[r]^-{x_{1,2}} & \bullet \ar[r]^-{x_{1,3}} & \cdots \ar[r]^-{x_{1,p_1-1}} &\bullet\\
0 \bullet\ar@<1ex>[r]^-{a} \ar@<-1ex>[r]_-{b} & \bullet 1  \ar[dr]_-{x_{N,1}} \ar[ur]^-{x_{1,1}}  &  \vdots & \vdots &\vdots &\vdots\\
 & &\bullet \ar[r]^-{x_{N,2}} & \bullet \ar[r]^-{x_{N,3}} & \cdots \ar[r]^-{x_{N,p_N-1}} & \bullet}
$$
modulo the set of relations
$$ x_{1, 1}a=0, \qquad x_{2,1}b=0, \qquad x_{i,1}(\lambda_i a-b)=0 \qquad \forall\; i=3, \ldots, N.$$

\vspace{.2in}

\section{Hall algebra of a weighted projective line}

\vspace{.1in}

From now on and unless specified, $k$ will be fixed to be a finite field, and we set $q= |k|, v=q^{-1/2}$.

\vspace{.1in}

\paragraph{\textbf{4.1.}} Let $\mathbf{H}_{\mathbb{X}}$ be the Hall algebra of the category $Coh(\mathbb{X})$. As a vector space, $\mathbf{H}_{\mathbb{X}}$ is the space of all functions $f : \mathcal{M} \to \C$ with finite support, where $\mathcal{M}$ is the set of all isomorphism classes of objects in $Coh(\mathbb{X})$. The multiplication is defined by the following formula (see e.g. \cite{LecturesI} for details)~: 
$$f \star g (\mathcal{G}) = \sum_{\mathcal{H} \subset \mathcal{G}} v^{\langle \mathcal{G}/\mathcal{H}, \mathcal{H} \rangle} f(\mathcal{G}/\mathcal{H}) g(\mathcal{H}).$$
Note that because of the support conditions on $f$ and $g$, only finitely many terms in the sum are nonzero.  The 
comultiplication is defined as
$$\Delta(f)(\mathcal{G},\mathcal{H})=\frac{v^{ \langle\mathcal{G},\mathcal{H}\rangle} }{|Ext^1(\mathcal{G},\mathcal{H})|}\sum_{\xi \in Ext^1(\mathcal{G},\mathcal{H})} f(X_{\xi})$$
where $X_{\xi}$ is the extension of $\mathcal{G}$ by $\mathcal{H}$ classified by $\xi$. Note that $\Delta$ takes values in a certain formal completion of $\mathbf{H}_{\mathbb{X}} \otimes \mathbf{H}_{\mathbb{X}}$ only.
The space $\mathbf{H}_{\mathbb{X}}$ carries a natural $K_0^+$-grading coming from the decomposition
$\mathcal{M}=\bigsqcup_{\alpha} \mathcal{M}_{\alpha}$. The triple $(\mathbf{H}_{\mathbb{X}}, *, \Delta)$ is a $K_0^+$-graded (twisted, topological) bialgebra. More precisely, we have
$$\Delta( f \star g)= \Delta(f) \star \Delta(g)$$
where the algebra structure on $\mathbf{H}_{\mathbb{X}} ^{\otimes 2}$ is defined as
$$(u_{\a} \otimes v_{\beta}) \star (u_{\alpha'} \otimes v_{\beta'})= v^{\langle \beta,\alpha'\rangle + \langle \alpha', \beta\rangle} (u_{\a} \star u_{\alpha'} \otimes v_{\beta} \star v_{\beta'})$$
for $u_{\gamma} \in \mathbf{H}_{\mathbb{X}}[\gamma]$ for any $\gamma \in \{\a,\a',\b,\b'\}$.
The last piece of structure which we will need is the Green pairing
\begin{equation}
\begin{split}
(\;|\;) ~: \mathbf{H}_{\mathbb{X}} \otimes \mathbf{H}_{\mathbb{X}} & \to \Q\\
f\otimes g &\mapsto \sum_{\mathcal{G}} \frac{f(\mathcal{G})g(\mathcal{G})}{|Aut(\mathcal{G})|}.
\end{split}
\end{equation}
It is a nondegenerate Hopf pairing, i.e $(a\star b, c)=(a \otimes b, \Delta(c))$.

\vspace{.1in}

\paragraph{\textbf{4.2.}} The product in the Hall algebra has the following geometric interpretation. Let $\mathbf{Coh}_{\a}$ be the stack classifying coherent sheaves on $\mathbb{X}$ of class $\a$. Let us also denote by $\widetilde{\mathbf{Coh}}_{\a,\b}$ the stack parametrizing inclusions $\mathcal{G} \subset \mathcal{H}$ of a coherent sheaf of class $\b$ in a coherent sheaf of class $\a+\b$. There is a convolution diagram
$$\xymatrix{ \mathbf{Coh}_{\a} \times \mathbf{Coh}_{\b} & \widetilde{\mathbf{Coh}}_{\a,\b} \ar[l]_-{q} \ar[r]^-{p} & \mathbf{Coh}_{\a+\b}}$$
in which the map $p$ associates to an inclusion $\mathcal{G}\subset \mathcal{H}$ the sheaf $\mathcal{H}$ while the
map $q$ associates to $\mathcal{G}\subset \mathcal{H}$ the pair of sheaves $(\mathcal{H}/\mathcal{G}, \mathcal{G})$. The map $p$ is proper while the map $q$ is a stack vector bundle of rank $-\langle \b, \a \rangle$. By definition, $\mathbf{H}_{\mathbb{X}}[\gamma]=Fun_0(\mathbf{Coh}_{\gamma}(k))$ is the space of function on $\mathbf{Coh}_{\gamma}(k)$ with finite support. Using these notations, the multiplication map $\star : \mathbf{H}_{\mathbb{X}}[\a] \otimes \mathbf{H}_{\mathbb{X}}[\b] \to \mathbf{H}_{\mathbb{X}}[\a+\b]$ is equal to 
$$f \star g= v^{-\langle \b,\a\rangle}p_*q^{*}(f \otimes g).$$

\vspace{.1in}

\paragraph{\textbf{4.3.}} For any $\gamma \in K_0^+$ we let
$$\m1_{\gamma}=1_{\mathcal{M}_{\gamma}}=\sum_{\substack{\mathcal{G} \in Coh(\mathbb{X})\\ \overline{\mathcal{G}}=\gamma}} 1_{\mathcal{G}}$$ 
be the characteristic function of the set of all sheaves of class $\gamma$. Unless $\gamma$ is a torsion class, $\m1_{\gamma}$ only belongs to the formal completion $\widehat{\mathbf{H}}_{\mathbb{X}}=\bigoplus_{\gamma}\widehat{\mathbf{H}}_{\mathbb{X}}[\gamma]$ where $\widehat{\mathbf{H}}_{\mathbb{X}}[\gamma]$ consists of all
functions $f : \mathcal{M}_{\gamma} \to \C$. 
For $\gamma \in K^+_0$ we let $\m1^{ss}_{\gamma}$ be the characteristic function of the set of semistable sheaves of class $\gamma$. This is a genuine element of $\mathbf{H}_{\mathbb{X}}$ because the set of such sheaves is finite. It can be shown (see \cite[Thm. 3.6]{Lin}) that $\m1^{ss}_{\gamma}\in \mathbf{H}^{sph}_{\mathbb{X}}$ for any $\gamma$. 

We define the \textit{spherical Hall algebra} $\mathbf{H}_{\mathbb{X}}^{sph}$ to be the subalgebra of $\mathbf{H}_{\mathbb{X}}$ generated by the collection of elements $\{\m1_{\gamma}\;|\; \gamma \in K_0^{+,tor}\}$ and $\{1_{\mathcal{L}}\;|\; \mathcal{L} \in Pic(\mathbb{X})\}$. It can be shown (see e.g. \cite[Section 4.6]{LecturesI} that $\mathbf{H}^{sph}_{\mathbb{X}}$ is a graded (twisted, topological) sub-bialgebra of $\mathbf{H}_{\mathbb{X}}$, and
that the restriction of the Green pairing to $\mathbf{H}^{sph}_{\mathbb{X}}$ is still nondegenerate. Of importance for us will be that the spherical Hall algebra admits an integral form in the following sense. Set
$$R=\C[t,t^{-1}], \qquad R_{loc}=R[(t^{l}-1)^{-1}\;|\; l \geq 1].$$
 
 The following is proved in \cite[Thm. 2.11]{Lin}~:
\vspace{.1in} 

\begin{theo}[Lin]\label{T:Lin} There exists a (twisted, topological) bialgebra ${}_R\mathbf{H}^{sph}_{\up}$ defined over $R$, generated by elements
$\{{}_R\m1_{\vec{u}}\;|\; \vec{u}\in L_{\up}\}$ and $\{\m1_{\gamma}\;|\; \gamma \in K_0^{+,tor}\}$, equipped with
a $R$-bilinear Hopf pairing 
$$(\;|\;)~:{}_R\mathbf{H}^{sph}_{\up} \otimes {}_R\mathbf{H}^{sph}_{\up} \to R_{loc}$$
satisfying the following property : for any finite field $k$, for any weighted projective line $\mathbb{X}$ of type $\up$ defined over $k$, the assignement
$t \mapsto |k|^{-1/2}, {}_R\m1_{\vec{u}} \mapsto \m1_{\mathcal{O}(\vec{u})}, {}_R\m1_{\gamma} \mapsto \m1_{\gamma}$
extends to a bialgebra morphism
$$ev_k~:{}_R\mathbf{H}^{sph}_{\up}\to \mathbf{H}^{sph}_{\mathbb{X}}.$$
Moreover, for any $\gamma \in K_0^+$ there exists an element ${}_R\m1^{ss}_{\gamma} \in {}_R\mathbf{H}^{sph}_{\up}$
such that for any $k, \mathbb{X}$ as above, $ev_k ({}_R\m1^{ss}_{\gamma})=\m1_{\gamma}^{ss}$.
\end{theo}

\vspace{.1in}

\paragraph{\textbf{4.3.}} Let us now fix a tilting sheaf $T$ and define $\mathcal{F}, \mathcal{T}$ as in Section~\textbf{3}. The next result is crucial for the proof of our main theorem.

\begin{prop}\label{T:1} For any $\ga \in K_0^+(\cT)$ we have $\mathbf{1}^{\cT}_{\ga} \in \mathbf{H}_{\mathbb{X}}^{sph}$. Moreover, there is an element ${}_R\mathbf{1}_{\ga}^{\cT} \in {}_R\mathbf{H}^{sph}_{\underline{p}}$ such that for any finite field $k$ and any weighted projective line $\mathbb{X}$ of type $\underline{p}$ defined over $k$ we have
$$\mathbf{1}^{\cT}_{\ga}= ev_k ({}_R\mathbf{1}_{\ga}^{\cT}).$$
The same holds for $\cF$ in place of $\cT$.
\end{prop}
\begin{proof} We will prove the statements concerning $\cT$ and $\cF$ simultaneously, by induction on the rank.
The most natural idea here would be to write an equality
\begin{equation}\label{E:eq01}
\m1_{\gamma} = \sum_{\a +\beta =\gamma} v^{ \langle \a, \beta\rangle} \m1^{\cF}_{\a} \m1^{\cT}_{\b}
\end{equation}
coming from the canonical filtration $\mathcal{G}^{\cT} \subset \mathcal{G}$ of a coherent sheaf $\mathcal{G} \in Coh(\mathbb{X})$. However, (\ref{E:eq01}) is only an equality in the completed Hall algebra and this leads to some
convergence issues. In order to avoid these difficulties, we fix an integer $n$ such that
\begin{equation}\label{E:eq00}
\forall\; \mathcal{G} \in \cT, \;\mu_{min}(\mathcal{G}) \geq n.
\end{equation}
and consider the full subcategory $Coh^{\geq n}$ of
$Coh(\mathbb{X})$ whose objects are coherent sheaves $\mathcal{G}$ satisfying the following property~: 
\begin{equation}\label{E:eq03}
n \leq \mu_{min}(\mathcal{G}).
\end{equation}
Let us set
$$\m1_{\gamma, \geq n}=\sum_{\substack{\mathcal{G} \in Coh^{\geq n} \\ \overline{\mathcal{G}}=\gamma}} 1_{\mathcal{G}}, \qquad
\m1^{\cT}_{\gamma, \geq n}=\sum_{\substack{\mathcal{G} \in Coh^{\geq n} \cap \cT \\ \overline{\mathcal{G}}=\gamma}} 1_{\mathcal{G}}, \qquad
\m1^{\cF}_{\gamma, \geq n}=\sum_{\substack{\mathcal{G} \in Coh^{\geq n}\cap \cF  \\ \overline{\mathcal{G}}=\gamma}} 1_{\mathcal{G}}.
$$
Observe that all the above sums are finite, i.e. $\m1_{\gamma, \geq n}, \m1^{\cT}_{\gamma, \geq n}, \m1^{\cF}_{\gamma,\geq n}$ all belong to $\mathbf{H}_{\mathbb{X}}$. Also, $\m1_{\gamma, \geq n}=0$ unless
$\mu(\gamma) \in \geq n$. The following lemma is a truncated version of (\ref{E:eq01}).

\vspace{.1in}

\begin{lem}\label{L:01} For any $\gamma \in K_0^+$ satisfying $\mu(\gamma) \geq n$ we have
\begin{equation}\label{E:eq04}
\m1_{\gamma,\geq n} = \sum_{\substack{\a +\beta =\gamma}} v^{\langle \a, \beta\rangle} \m1^{\cF}_{\a,\geq n} \m1^{\cT}_{\b,\geq n}.
\end{equation}
Moreover, the number of contributing terms in the sum on the right hand side of (\ref{E:eq04}) is finite.
\end{lem}
\begin{proof} The fact that the number of pairs $(\a,\b)$ such that $\a+\b=\gamma$ and $\mu(\a), \mu(\beta) \geq n$ is finite is a consequence of the fact that $K_0^+ \cap \{\sigma\;|\; \mu(\sigma) \in \geq n\}$ is a strictly convex subcone of $K_0^+$. In order to prove (\ref{E:eq04}) it is enough to show that if $\mathcal{G}$ is any coherent sheaf then $\mathcal{G} \in Coh^{\geq n}$ if and only if in the canonical exact sequence
$$\xymatrix{ 0 \ar[r] & \mathcal{G}^{\cT} \ar[r] & \mathcal{G} \ar[r] & \mathcal{G}^{\cF} \ar[r] & 0}$$
we have $\mathcal{G}^{\cT} \in Coh^{\geq n} \cap \cT$ and $\mathcal{G}^{\cF} \in Coh^{\geq n} \cap \cF$. 
Assume first that $\mathcal{G} \in Coh^{\geq n}$. By (\ref{E:eq00}) it suffices to check that $\mu_{min}(\mathcal{G}^{\cF}) \geq n$. This follows from the inequality
$$\mu_{min}(\mathcal{G}^{\cF}) \geq \mu_{min}(\mathcal{G}) \geq n$$
which in turn follow from the fact $\mathcal{G}^{\cF}$ is a quotient sheaf of $\mathcal{G}$. The other implication is obvious.
\end{proof}

We now proceed with the proof of Proposition~\ref{T:1}. Assume first that $\gamma$ is of rank $0$. If ${T}$ is a vector bundle then any torsion sheaf belongs to $\mathcal{T}$ and we have $\mathbf{1}_{\gamma}=\m1^{\cT}_{\gamma}$ so that the statement of Proposition~\ref{T:1} for $\gamma$ follows from Theorem~\ref{T:Lin}. Suppose that $T$ is not a vector bundle, so that a priori we may have
$\mathbf{1}_{\gamma} \neq \m1^{\cT}_{\gamma}$ for some rank zero class $\gamma$. We argue by induction on the degree of $\gamma$.
Fix some $d >0$ and assume that the statement of Proposition~\ref{T:1} is proved for all $\gamma$ of rank zero and degree at most
$d-1$. If $\gamma$ belongs to $K_0^+(\cT)$ then by (\ref{E:eq01}),
$$\m1_{\ga}= \m1^{\cT}_{\gamma} + \sum_{\substack{\a +\beta =\gamma\\ \a, \b \neq 0}} v^{ \langle \a, \beta\rangle} \m1^{\cF}_{\a} \m1^{\cT}_{\b}$$
and similarly if $\gamma$ belongs to $K_0^+(\cF)$ then 
$$\m1_{\ga}= \m1^{\cF}_{\gamma} + \sum_{\substack{\a +\beta =\gamma\\ \a, \b \neq 0}} v^{ \langle \a, \beta\rangle} \m1^{\cF}_{\a} \m1^{\cT}_{\b}.$$
We now define accordingly
$${}_R\m1^{\cT}_{\gamma}:= {}_R\m1_{\gamma}-  \sum_{\substack{\a +\beta =\gamma\\ \a, \b \neq 0}} v^{ \langle \a, \beta\rangle} {}_R\m1^{\cF}_{\a} {}_R\m1^{\cT}_{\b}$$
or
$${}_R\m1^{\cF}_{\gamma}:= {}_R\m1_{\gamma}-  \sum_{\substack{\a +\beta =\gamma\\ \a, \b \neq 0}} v^{ \langle \a, \beta\rangle} {}_R\m1^{\cF}_{\a} {}_R\m1^{\cT}_{\b}.$$

We give a similar proof for the case of classes $\gamma$ of nonzero rank, taking care to truncate (\ref{E:eq01}) using Lemma~\ref{L:01}.
Let us fix $n$ verifying (\ref{E:eq00}). We will prove the statement of the proposition for the elements $\m1^{\cT}_{\ga, \geq n}, \m1^{\cF}_{\ga,\geq n}$ in place of $\m1^{\cT}_{\ga}, \m1^{\cF}_{\ga}$ respectively. Since, by Lemma~\ref{L:finitesum}, for any fixed $\ga$ we have
$\m1^{\cT}_{\ga, \geq n}=\m1^{\cT}_{\ga}$ and $\m1^{\cF}_{\ga, \geq n}=\m1^{\cF}_{\ga}$ for $n \ll 0$ the theorem will follow.
We will argue by induction on the rank. Let us fix some $r >0$ and let us assume that the statement of the theorem is true for all $\m1^{\cT}_{\sigma, \geq n}, \m1^{\cF}_{\sigma,\geq n}$ with $\sigma$ satisfying $rank(\sigma)<r$. Choose some $\gamma \in K_0^+(\cT), \mu(\gamma) \geq n$ and $rank(\gamma) =r$.
By Lemma~\ref{L:01} we have
$$\m1_{\ga, \geq n}= \m1^{\cT}_{\gamma} + \sum_{\substack{\a +\beta =\gamma\\ \a, \b \neq 0}} v^{ \langle \a, \beta\rangle} \m1^{\cF}_{\a,\geq n} \m1^{\cT}_{\b,\geq n}.$$
For any pair $(\a,\b)$ occuring in the above sum we have $rank(\a), rank(\b) < r$. Hence by the induction hypothesis the elements
$\m1^{\cF}_{\a,\geq n}$ and $\m1^{\cT}_{\b,\geq n}$ admit integral lifts ${}_R\m1^{\cF}_{\a,\geq n}$ and ${}_R\m1^{\cT}_{\b,\geq n}$ in ${}_R\mathbf{H}^{sph}_{\underline{p}}$. By \cite[Section. 3.3.]{Lin}, the characteristic functions
of all the Harder-Narasimhan strata in $\mathbf{Coh}_{\mathbb{X}}$ admit integral lifts. Since $\m1_{\gamma, \geq n}$ is a sum of such characteristic functions, it follows that there exists a lift ${}_R\m1_{\gamma, \geq n}$ 
of $\m1_{\gamma, \geq n}$. But then the element
$${}_R\m1^{\cT}_{\gamma, \geq n}:= {}_R\m1_{\gamma, \geq n}-  \sum_{\substack{\a +\beta =\gamma\\ \a, \b \neq 0}} v^{ \langle \a, \beta\rangle} {}_R\m1^{\cF}_{\a,\geq n} {}_R\m1^{\cT}_{\b,\geq n}$$
is a lift of $\m1^{\cT}_{\gamma, \geq n}$ to ${}_R\mathbf{H}^{sph}_{\underline{p}}$. The same argument works
for a class $\gamma \in K_0^+(\cF)$. Proposition~\ref{T:1} is proved.
\end{proof}

\vspace{.1in}

\noindent
\textbf{Remark.} From the above proof, we see that the integral lifts ${}_R\m1^{\cT}_{\gamma}, {}_R\m1^{\cF}_{\gamma}$ only depend on the combinatorial data $(\up, K_0^+(\mathcal{F}), K_0^+(\mathcal{T}))$.

\vspace{.1in}

\section{Kac polynomials for weighted projective lines}

\vspace{.1in}

\paragraph{\textbf{5.1.}}
We will say that a coherent sheaf $\mathcal{F} \in Coh(\mathbb{X})$ is \textit{geometrically indecomposable} if it is indecomposable and if $\mathcal{F} \otimes \overline{k}$ is also indecomposable. The number of geometrically indecomposable coherent sheaves (counted up to isomorphism) of a given class
$\alpha \in K_0^+(Coh(\mathbb{X}))$ is finite. We will denote it by $A_{\alpha}(\mathbb{X})$. The following Theorem
is proved in \cite[Thm. 7.1]{SHiggs}.

\vspace{.1in}

\begin{theo} For any $\alpha \in K^+_0(Coh(\mathbb{X}))$ and any tuple $\underline{p}=(p_1, \ldots, p_N)$ there is a polynomial
$A_{\underline{p}, \alpha} \in \mathbb{Q}[x]$ such that for any finite field $k$ and any weighted projective line 
$\mathbb{X}=\mathbb{X}_{\underline{\lambda},\underline{p}}$ defined over $k$ we have
$$A_{\alpha}(\mathbb{X})=A_{\underline{p}, \alpha}(|k|).$$
\end{theo}

\vspace{.1in}

\paragraph{\textbf{5.2.}} In this section, we prove some variants of the above result. Namely, denote by $A_{\alpha}^{\mathcal{T}}(\mathbb{X})$ and
$A_{\alpha}^{\mathcal{F}}(\mathbb{X})$ the number of geometrically indecomposable sheaves in the subcategory $\mathcal{T}$ (resp. $\mathcal{F}$) of class $\alpha$.

\begin{theo}\label{T:2} For any $\alpha \in K^+_0(Coh(\mathbb{X}))$ and any tuple $\underline{p}=(p_1, \ldots, p_N)$ there are  polynomials
$A^{\mathcal{T}}_{\underline{p}, \alpha},A^{\mathcal{F}}_{\underline{p}, \alpha}  \in \mathbb{Q}[x]$ such that for any finite field $k$ and any weighted projective line 
$\mathbb{X}=\mathbb{X}_{\underline{\lambda},\underline{p}}$ defined over $k$ we have
$$A^{\mathcal{T}}_{\alpha}(\mathbb{X})=A^{\mathcal{T}}_{\underline{p}, \alpha}(|k|), \qquad A^{\mathcal{F}}_{\alpha}(\mathbb{X})=A^{\mathcal{F}}_{\underline{p}, \alpha}(|k|).$$
\end{theo}
\begin{proof} The proof essentially follows the line of \cite[Thm. 1.1]{SHiggs}. We sketch the argument here and refer to \textit{loc. cit} for the details.
The idea is to relate, through some generating series, the number of absolutely indecomposable sheaves in $\mathcal{T}$  (of class $\alpha$, for all classes $\alpha$) to the volume of stacks parametrizing sheaves in $\mathcal{T}$ (of class $\alpha$, for all $\alpha$) equipped with a nilpotent endomorphism. We then express the volumes of these stacks in terms of the values of 
some pairings in the integral form ${}_R\mathbf{H}^{sph}_{\up}$ of the spherical Hall algebra $\mathbf{H}^{sph}_{\mathbb{X}}$. This yields the polynomiality of these volumes --and thus of the numbers $A_{\a}(\mathbb{X})$ and provides, in principle, a way to explicitly compute them.

\vspace{.1in}

We will deal with the category $\mathcal{T}$, the case of $\mathcal{F}$ being similar. Let us denote by $\mathbf{Coh}^{\cT}$ the stack parametrizing
objects of $\cT$ of class $\a$. This is an open substack of the (smooth) stack $\mathbf{Coh}_{\a}$ parametrizing objects
of $Coh(\mathbb{X})$ of class $\a$. Note that, unlike $\mathbf{Coh}_{\a}$, the stack $\mathbf{Coh}^{\cT}_{\a}$ is of finite type. Let $\mathbf{Nil}^{\cT}_{\alpha}$ denote the stack parametrizing pairs $(\mathcal{F}, \theta)$ with
$\mathcal{F} \in \mathbf{Coh}^{\cT}_{\a}$ and $\theta \in End(\mathcal{F})$ a nilpotent endomorphism. This is also 
a stack of finite type, but is in general singular. The following result essentially follows from the Krull-Schmitt property of $\cT$ and is proved in exactly the same way as \cite[Prop. 2.2]{SHiggs}~:

\vspace{.1in}

\begin{prop} The following relation holds in the completed group ring $\Q[[z^{\a}]]_{\a \in K^+_0}$~:
\begin{equation}\label{E:eq1}
\sum_{\alpha \in K_0^+} vol(\mathbf{Nil}^{\cT}_\a) z^{\a} = \text{exp}\left( \sum_{\substack{l \geq 1 \\ \a \in K_0^+}} \frac{1}{l(q^l-1)} A^{\cT}_{\a}(\mathbb{X} \otimes \mathbb{F}_{q^l}) z^{l\a}\right).
\end{equation}
\end{prop}

The computation of the volume of the stack $\mathbf{Nil}^{\cT}_{\a}$ proceeds along the following lines. We consider a stratification 
$$\mathbf{Nil}^{\cT}_{\a}= \bigsqcup_{\underline{\a}}\mathbf{Nil}^{\cT}_{\underline{\a}}$$
by the Jordan type of the nilpotent endomorphism $\theta$, where the index runs over all sequences $\underline{\a}=(\a_1, \ldots, \a_l)$ with $\a_i \in K_0^+, \sum_i i \a_i=\a$.  More precisely, to $(\mathcal{F},\theta) \in \mathbf{Nil}^{\cT}_\a$ we associate the sequence of surjections
\begin{equation}\label{E:seqsurj}
\mathcal{F}/ Im (\theta) \stackrel{d_1}{\tto} Im(\theta) / Im( \theta^2) \stackrel{d_2}{\tto} Im(\theta^2)/ Im(\theta^3)\cdots
\end{equation}
and we define the Jordan type $J(\mathcal{F}, \theta):=\underline{\a}=(\a_1, \a_2, \ldots)$ by
$$\alpha_i= \overline{Ker(d_i)}=\overline{Im(\theta^{i-1})/ Im(\theta^i)} - \overline{{Im}(\theta^i)/ Im(\theta^{i+1})}.$$
Note that $\sum_i i \alpha_i=\a$ (see \cite[Section~3.1]{SHiggs}). Given a tuple $\underline{\a}$ such that $\sum_i i \a_i=\a$ we set
$$\mathbf{Nil}^{\cT}_{\underline{\a}}=\{ (\mathcal{F},\theta) \in \mathbf{Nil}^{\cT}_{\a}\;|\; J(\mathcal{F},\theta)= \underline{\a}\}.$$

Let us now consider the stack $\widetilde{\mathbf{Coh}}^{\cT}_{\underline{\a}}$ parametrizing sequences of
surjective maps
\begin{equation}\label{E:seqsurj2}
\mathcal{G}_1 \stackrel{d_1}{\tto} \mathcal{G}_2 \stackrel{d_2}{\tto} \cdots
\end{equation}
of objects in $\cT$ such that $\overline{Ker(d_i)}=\a_i$. 
There is a canonical morphism of stacks $\pi_{\underline{\a}} : \mathbf{Nil}^{\cT}_{\underline{\a}} \to \widetilde{\mathbf{Coh}}^{\cT}_{\underline{\a}}$ which assigns to $(\mathcal{F},\theta)$ the sequence (\ref{E:seqsurj}). Observe that as $\cT$ is stable under quotients, all the sheaves $Im(\theta^i)$ and $\Im(\theta^{i})/Im(\theta^{i+1})$ indeed belong to $\cT$.

The proof of the following proposition is entirely analogous to that of \cite[Prop. 3.1.]{SHiggs}, see also \cite[Prop.~5.1]{MozSchiff}.

\vspace{.1in}

\begin{prop} The map $\pi_{\underline{\a}}$ is a stack vector bundle of rank 
$$r_{\underline{\a}}=-\bigg\{\sum_{i}(i-1) \langle \alpha_i, \alpha_i \rangle
+ \sum_{i<j} i (\alpha_i, \alpha_j) \bigg\} + \sum_{i<j} \langle \alpha_j, \alpha_i \rangle$$
In particular,
\begin{equation}\label{E:eq2}
vol(\mathbf{Nil}^{\cT}_{\underline{\a}})=q^{r(\underline{\a})} vol(\widetilde{\mathbf{Coh}}^{\cT}_{\underline{\a}}).
\end{equation}
\end{prop}

The key point here is that the volume of $\widetilde{\mathbf{Coh}}^{\cT}_{\underline{\a}}$ can be expressed
as a certain pairing in the \textit{spherical} Hall algebra $\mathbf{H}_{\mathbb{X}}^{sph}$. Indeed, recall the diagram
$$\xymatrix{ \mathbf{Coh}_{\a_s} \times \cdots \times \mathbf{Coh}_{\a_1} & \widetilde{\mathbf{Coh}}_{\underline{\a}} \ar[l]_-{q} \ar[r]^-{p} & \mathbf{Coh}_{\gamma}}
$$
occuring in the definition of the product in $\mathbf{H}_{\mathbb{X}}^{sph}$, where $\gamma=\sum_i \a_i$. Then $\widetilde{\mathbf{Coh}}^{\cT}_{\underline{\a}}= p^{-1} (\mathbf{Coh}^{\cT}_{\gamma})$ and
by definition of the Green pairing
\begin{equation}\label{E:eq3}
vol(\widetilde{\mathbf{Coh}}^{\cT}_{\underline{\a}})= \langle p_!(\mathbf{1}_{\widetilde{\mathbf{Coh}}^{\cT}_{\underline{\a}}}) \;|\; \mathbf{1}_{\gamma}\rangle =
\langle p_! (\mathbf{1}_{\widetilde{\mathbf{Coh}}_{\underline{\a}}}) \;|\; \mathbf{1}^{\cT}_{\gamma}\rangle
=q^{-\frac{1}{2}\sum_{i<j}\langle \a_i, \a_j\rangle}\langle \mathbf{1}_{\a_s} \cdots \mathbf{1}_{\a_1} \;|\;   \mathbf{1}^{\cT}_{\gamma}\rangle.
\end{equation}
Since $\mathbf{1}^{\cT}$ and the $\mathbf{1}_{\a_i}$ are, by Proposition~\ref{T:1} specializations of elements in the generic spherical Hall
algebra ${}_R\mathbf{H}^{sph}_{\underline{p}}$, it follows that there exists rational functions $f_{\underline{\a}} \in \Q(t)$ such that for any finite field $k$ and any weighted projective line $\mathbb{X}$ of type $\underline{p}$ defined over $k$,
$$vol(\mathbf{Nil}^{\cT}_{\underline{\a}})=ev_k (f_{\underline{\alpha}}).$$

Combining equations (\ref{E:eq1}), (\ref{E:eq2}) and (\ref{E:eq3}) we obtain the following relation, for a weighted projective line $\mathbb{X}$ defined over $\mathbb{F}_q$~:
\begin{equation}\label{E:eq4}
 \sum_{\substack{l \geq 1 \\ \a \in K_0^+}} \frac{1}{l(q^l-1)} A^{\cT}_{\a}(\mathbb{X} \otimes \mathbb{F}_{q^l}) z^{l\a}
 =ev_{\mathbb{F}_q} log \left( \sum_{\underline{a}} f_{\underline{\a}}(t)z^{\sum_i i \a_i}\right)
 \end{equation}  
 We deduce as wanted that $A_{\a}^{\cT}(\mathbb{X})$ is the evaluation at $t=v$ of some rational function $A^{\cT}_{\underline{p},\a}(t)$, and that in fact
 $$\sum_{\a} A^{\cT}_{\underline{p},\a}(t)z^\a=(t-1) \text{Log}\left(\sum_{\underline{\a}} f_{\underline{\a}}(t)z^{\sum_i i \a_i} \right).$$
 To finish, note that the rational function $A^{\cT}_{\underline{p},\a}(t)$ is necessarily in $\Q[t^2]$ as it takes integral values when evaluated at all prime powers $t=q^{l/2}$. This proves theorem~\ref{T:2} for the category $\cT$.
\end{proof}

\vspace{.1in}

In view of what is known for usual Kac polynomials, the following conjecture seems natural :

\begin{conj} For any $\up$ and any $\a$ we have $A^{\cT}_{\up,\a}(x), A_{\up,\a}^{\cF} \in \N[x]$.
\end{conj}

\vspace{.2in}

\paragraph{\textbf{5.3.}} \textit{Proof of Theorem~\ref{T:main2}.} By Theorem~\ref{T: indecan} i), we have $A_{\mathbf{d},\up}=A^{\cT}_{\psi(\mathbf{d}),\up}$ if
$\psi(\mathbf{d}) \in K_0^+(\cT)$ while  $A_{\mathbf{d},\up}=A^{\cF}_{\psi(\mathbf{d}),\up}$ if
$\psi(\mathbf{d}) \in -K_0^+(\cF)$. Note that by Lemma~\ref{L:intersect}, the two cases are mutually exclusive.
 By Theorem~\ref{T:2}, these polynomials belong to
$\Q[x]$. It only remains to prove that they belong to $\Z[x]$. For this, we will use the following result due to Katz, see \cite[Sect. 6]{HRV}~:

\vspace{.05in}

\begin{lem}[Katz]\label{L:Katz} Let $Z/\mathbb{F}_q$ be a constructible set defined over $\mathbb{F}_q$. Assume that there exists a polynomial $P_Z(x) \in \mathbb{C}[x]$ such that for any
field extension $\mathbb{F}_{q^l} / \mathbb{F}_q$, $|Z(\mathbb{F}_{q^l})|=P_Z(q^l)$. Then $P_Z(x) \in \Z[x]$.  
\end{lem}

\vspace{.05in}

In order to use the above result, we need to relate the polynomials $A_{\bd, C}$ to the point count of some constructible sets. Fix $\bd$ and set 
$$\text{Aut}_{\bd,C}=\{(M,g)\;|\; g \in Aut_{Rep_C}(M)\} \subset \text{Rep}_{\bd,C} \times GL_{\bd},$$
$$\text{Aut}^{abs. ind.}_{\bd,C}=\{(M,g)\;|\; M\; \text{is\;abs.\;indec.}\} \subset \text{Aut}_{\bd,C}.$$
Then $\text{Aut}_{\bd,C}$ is an algebraic variety defined over $\mathbb{F}_q$ while $\text{Aut}^{abs. ind.}_{\bd,C}$ is a constructible subset of $\text{Aut}_{\bd,\C}$. Notice that $\text{Aut}^{abs. ind.}_{\bd,C}$ is compatible with field
extensions, i.e. 
$$\text{Aut}^{abs. ind.}_{\bd,C} (\mathbb{F}_{q^l})=\{(M, g)\in \text{Rep}_{\bd, C}(\mathbb{F}_{q^l}) \times GL_{\bd}(\mathbb{F}_{q^l})\;|\; M \;\text{is\;abs.\;indec}, g \in Aut(M)\}$$
(this property is not true if one replaces `absolutely indecomposable' by `indecomposable').
Hence, for any $l \geq 1$, we have
$$\frac{|\text{Aut}^{abs. ind.}_{\bd,C}(\mathbb{F}_{q^l})|}{|GL_{\bd}(\mathbb{F}_{q^l})|}=|\{M \in Rep_{\bd,C}(\mathbb{F}_{q^l})\;|\; M\;\text{is\;abs.\;indec}\}/\sim|=A_{\bd,C}(q^l).$$
By Lemma~\ref{L:Katz} above we deduce that $A_{\bd,C}(x) \prod_{i}\prod_{k=0}^{d_i-1} (x^{d_i}-x^k) \in \Z[x]$.
From the facts that $A_{\bd,C}(x) \in \mathbb{Q}[x]$ and $ \prod_{i}\prod_{k=0}^{d_i-1} (x^{d_i}-x^k)$ is monic
we deduce that $A_{\bd,C}(x) \in \Z[x]$ as wanted. Theorem~\ref{T:main2} is proved.

\vspace{.2in}

\section{Volume of the stacks $\mathbf{Rep}_{\mathbf{d}}C$.}

\vspace{.1in}

The aim of this short section is to show, using the same kind of arguments as above, that the volumes of the stacks
of representations of (almost, concealed) canonical algebras are also given by some rational function in the size of the finite field.

\vspace{.1in}

\paragraph{\textbf{6.1.}} Let us denote by $\textbf{Rep\,}C_T$ the stack parametrizing (finite-dimensional) representations of $C$. This stack can be presented as a disjoint union
$$\textbf{Rep}\,C_T=\bigsqcup_{\bd} \textbf{Rep}_{\bd}\,C_T, \qquad 
\textbf{Rep}_{\bd}\,C_T = [E_{\bd,C_T}/ G_{\bd}]$$
where $E_{\bd,C_T}$ is a certain affine algebraic variety determined from a presentation of $C_T$ as the path algebra of a quiver $Q$ with relations, and where $\bd$ runs among all dimension vectors of $Q$. Recall that $K_0(\mathcal{T})^{+}$, resp. $K_0(\mathcal{F}[1])^+$ is the submonoid of  $K_0(C_T)^+$ generated by the classes of objects in $\mathcal{T}$, resp. $\mathcal{F}[1]$. There is a partition
$$\textbf{Rep}_{\bd}\,C_T = \bigsqcup_{\bd_1, \bd_2} \textbf{Rep}_{\bd_1,\bd_2}\,C_T$$
where the union runs over all pairs $(\bd_1, \bd_2)$ such that  $\bd_1 \in K_0(\mathcal{F}[1])^+, \bd_2 \in K_0(\mathcal{T})^+$ and $\bd_1+\bd_2=\bd$, and where $\textbf{Rep}_{\bd_1, \bd_2}\,C_{T}$ is the locally closed substack of $\textbf{Rep}_{\bd}\,C_T$ parametrizing modules $M$ whose canonical submodule $M_1 \in \mathcal{F}[1]$ is of dimension $\bd_1$. By Theorem~\ref{T: indecan} iii), the canonical morphism
$$ \textbf{Rep}_{\bd_1,\bd_2}\,C_T \to  \textbf{Rep}_{\bd_1,0}\,C_T \times  \textbf{Rep}_{0,\bd_2}\,C_T, \qquad M \mapsto (M^{(1)}, M^{(2)})$$
is a stack bundle of dimension $-\langle \bd_2, \bd_1\rangle$, see e.g. \cite[Sect. ~5]{MozSchiff}. It follows that, when $k=\mathbb{F}_q$ is a finite field,
$$vol(\textbf{Rep}_{\bd}\,C_T)=\sum_{\bd_1, \bd_2} q^{-\langle \bd_2,\bd_1\rangle} vol(\textbf{Rep}_{\bd_1,0}\,C_T) vol(\textbf{Rep}_{0,\bd_2}\,C_T)$$
where the sum ranges over the same set of pairs $(\bd_1,\bd_2)$ as above. Note that this sum is finite.

\vspace{.1in}

\paragraph{\textbf{6.2.}} The substacks $\textbf{Rep}_{0,\bd_2}\,C_T$ and $\textbf{Rep}_{\bd_1,0}\,C_T$ are respectively open in $\textbf{Rep}_{\bd_2}\,C_T$ and  in $\textbf{Rep}_{\bd_1}\,C_T$. Moreover, the derived
equivalence $D^b(Coh(\mathbb{X})) \stackrel{\sim}{\to} D^b(mod(C_T))$ induces isomorphism of stacks
$$\textbf{Coh}^{\mathcal{T}}_{\psi^{-1}(\bd_2)}(\mathbb{X}) \simeq \textbf{Rep}_{0,\bd_2}\,C_T, \qquad
\textbf{Coh}^{\mathcal{F}}_{\psi^{-1}(-\bd_1)}(\mathbb{X}) \simeq \textbf{Rep}_{\bd_1,0}\,C_T$$
and therefore
$$vol(\textbf{Rep}_{\bd_1,0}\,C_T)=vol(\textbf{Coh}^{\mathcal{F}}_{\psi^{-1}(-\bd_1)}(\mathbb{X}))=\langle \mathbf{1}^{\mathcal{F}}_{\psi^{-1}(-\bd_1)} \,|\, \mathbf{1}^{\mathcal{F}}_{\psi^{-1}(-\bd_1)}\rangle, $$
$$vol(\textbf{Rep}_{0,\bd_2}\,C_T)=vol(\textbf{Coh}^{\mathcal{T}}_{\psi^{-1}(\bd_2)}(\mathbb{X}))=\langle \mathbf{1}^{\mathcal{T}}_{\psi^{-1}(\bd_2)} \,|\, \mathbf{1}^{\mathcal{T}}_{\psi^{-1}(\bd_2)}\rangle.$$
By Proposition~\ref{T:1}, the elements $\mathbf{1}^{\mathcal{T}}_{\gamma}, \mathbf{1}^{\mathcal{F}}_{\gamma}$ admit lifts
${}_R\mathbf{1}^{\mathcal{T}}_{\gamma}, {}_R\mathbf{1}^{\mathcal{F}}_{\gamma}$ to the integral form of the spherical
Hall algebra ${}_R\mathbf{H}^{sph}_{\underline{p}}$. We deduce that there exists for any $\gamma \in K_0$ elements $vol^{\mathcal{T}}_{\gamma}, vol^{\mathcal{F}}_{\gamma} \in R_{loc}$ such that for any finite field $\mathbb{F}_q$, for any weighted projective line $\mathbb{X}$ defined over $\mathbb{F}_q$ of type $\up$ we have
$$vol(\textbf{Coh}^{\mathcal{T}}_{\gamma}(\mathbb{X}))=vol^{\mathcal{T}}_{\gamma}(q), \qquad 
vol(\textbf{Coh}^{\mathcal{F}}_{\gamma}(\mathbb{X}))=vol^{\mathcal{F}}_{\gamma}(q).$$
The following result is an easy consequence of the above considerations~:

\vspace{.05in}

\begin{theo} For any combinatorial type $\mathbf{c}=(\up, N_1, N_2)$, for any $\mathbf{d} \in N_{\up}$ there exists an element $V_{\mathbf{c},\bd} \in R_{loc}$ such that for any almost concealed canonical algebra of combinatorial type $\mathbf{c}$ defined over a finite field $k$ we have $vol(\textbf{Rep}_{\bd}\,C)=V_{\mathbf{c},\bd}(|k|)$.
\end{theo}

\vspace{.1in}

%
%
%
%
%

\vspace{.2in}

\noindent
Laboratoire de Math\'ematiques d'Orsay, Univ. Paris-Sud, CNRS,\\
 Universit\'e Paris-Saclay, 91405 Orsay, France.\\
email~: \texttt{olivier.schiffmann@math.u-psud.fr}, \texttt{pierre-guy.plamondon@math.u-psud.fr}

\end{document}